\documentclass[11pt, reqno,oneside]{amsart}   	
\usepackage{geometry}                		
\usepackage{graphicx}				
\usepackage{color}								
\usepackage{amssymb}
\usepackage{amsthm}
\usepackage{amsmath}
\usepackage{hyperref}
\usepackage{mathrsfs}
\usepackage{mathtools}
\usepackage[mathscr]{eucal}
\usepackage[T1]{fontenc}
\usepackage{hyperref}
\usepackage{url}
\usepackage{breakurl}

\newcommand{\bburl}[1]{\textcolor{blue}{\url{#1}}}

\allowdisplaybreaks
\newcommand{\Bb}{\mathscr{B}}

\newcommand{\lp}{\left(}
\newcommand{\rp}{\right)}
\newcommand{\A}{\mathscr{A}}
\newcommand{\pez}[1]{\left( #1\right)}
\newcommand{\abs}[1]{\left|#1\right|}
\newcommand{\Hp}{\mathfrak{H}}
\newcommand{\R}{\mathbf{R}}
\newcommand{\C}{\mathbf{C}}

\newcommand{\sP}{\mathscr{P}}
\newcommand{\overbar}[1]{\mkern 1.5mu\overline{\mkern-1.5mu#1\mkern-1.5mu}\mkern 1.5mu}

\newcommand{\kkot}[1]{ \frac{\sin \pi {#1} }{\pi {#1} } }
\newcommand{\foh}{\frac12}

\newcommand{\eeqno}{\tag{\stepcounter{equation}\theequation}}

\DeclareMathOperator{\supp}{supp}
\newcommand{\oo}{\infty}
\newcommand{\ra}{\rightarrow}
\newcommand{\vep}{\varepsilon}
\renewcommand{\geq}{\geqslant}
\renewcommand{\leq}{\leqslant}
\renewcommand{\ge}{\geqslant}
\renewcommand{\le}{\leqslant}

\makeatletter
\newtheorem*{rep@theorem}{\rep@title}
\newcommand{\newreptheorem}[2]{%
\newenvironment{rep#1}[1]{%
 \def\rep@title{#2 \ref{##1}}%
 \begin{rep@theorem}}%
 {\end{rep@theorem}}}
\makeatother

\theoremstyle{plain}
\newtheorem{theorem}{Theorem}[section]
\newreptheorem{theorem}{Theorem}
\newtheorem{lemma}[theorem]{Lemma}
\newreptheorem{lemma}{Lemma}

\newtheorem{proposition}[theorem]{Proposition}
\theoremstyle{definition}
\newtheorem{definition}[theorem]{Definition}
\theoremstyle{remark}
\newtheorem{remark}[theorem]{Remark}
\theoremstyle{plain}

\newcommand{\reff}[1]{\hyperref[#1]{\ref{#1}}}

\numberwithin{equation}{section}

\title{One-Level density for holomorphic cusp forms of arbitrary level}

\author[O. Barrett]{Owen Barrett}
\address{Department of Mathematics, University of Chicago, Chicago, IL, 60637, USA}
\email{barrett@math.uchicago.edu}

\author[P. Burkhardt]{Paula Burkhardt}
\address{Department of Mathematics, University of California, Berkeley, Berkeley, CA, 94720, USA}
\email{paulab@math.berkeley.edu, peb02012@pomona.edu}

\author[J. DeWitt]{Jonathan DeWitt}
\address{Department of Mathematics and Statistics, Haverford College, Haverford, PA 19041, USA}
\email{jdewitt@haverford.edu, jon.dewitt@gmail.com}

\author[R. Dorward]{Robert Dorward}
\address{Department of Mathematics, Oberlin College, Oberlin, OH 44074, USA}
\email{rdorward@oberlin.edu, bobbydorward@gmail.com}

\author[S.\,J. Miller]{Steven J. Miller}
\address{Department of Mathematics \& Statistics, Williams College, Williamstown, MA 01267, USA}
\email{sjm1@williams.edu, Steven.Miller.MC.96@aya.yale.edu}

\date{\today}							

\subjclass[2010]{11M26 (primary), 11M41, 15A52 (secondary).}

\keywords{Low lying zeroes, one level density, cuspidal newforms, Petersson formula.}

\thanks{The first four named authors were supported by NSF grant DMS1347804  and Williams College; the fifth-named author was partially supported by NSF grants DMS0850577 and DMS1561945. We thank Jim Cogdell for discussions on the signs of functional equation, Djorde Mili\'cevi\'c and Valentin Blomer for conversations on their work, and the referees and editor for many helpful comments and observations.}

\date{\today}

\begin{document}
\maketitle

\begin{abstract}
In 2000 Iwaniec, Luo, and Sarnak proved for certain families of $L$-functions associated to holomorphic newforms of square-free level that, under the Generalized Riemann Hypothesis, as the conductors tend to infinity the one-level density of their zeros matches the one-level density of eigenvalues of large random matrices from certain classical compact groups in the appropriate scaling limit. We remove the  square-free restriction by obtaining a trace formula for arbitrary level by using a basis developed by Blomer and Mili\'cevi\'c, which is of use for other problems as well.
\end{abstract}

\section{Introduction}

Montgomery \cite{Montgomery} conjectured that the pair correlation of
critical zeros up to height $T$ of the Riemann zeta function $\zeta(s)$ coincides with the
pair correlation of eigenvalues of random unitary matrices of dimension $N$ in the appropriate
limit as $T,N\ra\oo$. This remarkable connection initiated a new branch of number theory concerned with
relating the statistics of zeros of $\zeta(s)$, and of $L$-functions more generally,
to those of eigenvalues of random matrices. While additional support for this agreement was obtained by the work of Hejhal \cite{Hej} on the triple correlation of $\zeta(s)$, Rudnick and Sarnak \cite{RS} on the $n$-level correlation for cuspidal automorphic forms, and Odlyzko \cite{Od1, Od2} on the spacings between adjacent zeros of $\zeta(s)$, the story cannot end here as these statistics are insensitive to the behavior of any finite set of zeros. As the zeros at and near the central point play an important role in a variety of problems, this led Katz and Sarnak \cite{KS1,KS2} to develop a new statistic which captures this behavior.

\begin{definition} Let $L(s,f)$ be an $L$-function with zeros in the critical strip $\rho_f = 1/2 + i \gamma_f$ (note $\gamma_f \in \mathbb{R}$ if and only if the Generalized Riemann Hypothesis holds for $f$), and let $\phi$ be an even Schwartz function whose Fourier transform has compact support. The \textbf{one-level density} is \begin{equation}\label{one level density} D_1(f;\phi)\ := \ \sum_{\rho_f} \phi\left(\frac{\gamma_f}{2\pi}\log c_f\right),\end{equation} where $c_f$ is the analytic conductor.
\end{definition}

Their Density Conjecture \cite{KS1, KS2} states that the scaling limits of eigenvalues of classical compact groups near 1 correctly model the behavior of these zeros a family $\mathcal{F}$ of $L$-functions as the conductors tend to infinity. Specifically, let $\mathcal{F}_N$ be a sub-family of $\mathcal{F}$ with suitably restricted conductors; often one takes all forms of conductor $N$, or conductor at most $N$, or conductor in the range $[N, 2N]$. If the symmetry group is $\mathcal{G}$, then we expect \begin{equation} \mathcal{D}_1(\mathcal{F};\phi) \ := \ \lim_{N\to\infty} \frac1{|\mathcal{F}_N|} \sum_{f\in \mathcal{F}_N} D_1(f;\phi) \ = \ \int_{-\infty}^\infty \phi(x) W_1(\mathcal{G})(x) dx  \ = \ \int_{-\infty}^\infty \widehat{\phi}(t) \widehat{W}_1(\mathcal{G})(t) dt, \end{equation} where $K(y) = \kkot{y}$, $ K_\epsilon(x,y) = K(x-y) + \epsilon K(x+y)$ for $\epsilon = 0, \pm 1$, and
\begin{eqnarray}\label{eqdensitykernels}
W_1(\mathrm{SO(even)})(x) &\ =\ &  K_1(x,x) \nonumber\\ W_1(\mathrm{SO(odd)})(x) & = & K_{-1}(x,x)  + \delta_0(x) \nonumber\\ W_1(\mathrm{O})(x) & =
& \foh W_1(\mathrm{SO(even)})(x) + \foh W_1(\mathrm{SO(odd)})(x) \nonumber\\ W_1(\mathrm{U})(x) & =
& K_0(x,x) \nonumber\\ W_1(\mathrm{Sp})(x) &=&  K_{-1}(x,x).
\end{eqnarray}

While the Fourier transforms of the densities of the orthogonal groups all equal $\delta_0(y) + 1/2$ in $(-1,1)$, they are mutually distinguishable for larger support (and are distinguishable from the unitary and symplectic cases for any support). There is now an enormous body of work showing the 1-level densities of many families (such as Dirichlet $L$-functions, elliptic curves, cuspidal newforms, Maass forms, number field $L$-functions, and symmetric powers of ${\rm GL}_2$ automorphic representations) agree with the scaling limits of a random matrix ensemble; see \cite{AAILMZ,AM,DM1,FiMi,FI,Gao,GK,Gu,HM,HR,ILS,KS1,KS2,Mil,MilPe,OS1,OS2,RR,Ro,Rub1,Rub2,ShTe,Ya,Yo} for some examples, and \cite{DM1,DM2,ShTe} for discussions on how to determine the underlying symmetry. For additional readings on connections between random matrix theory, nuclear physics and number theory see \cite{BFMT-B, Con, CFKRS, FM, For, KeSn1, KeSn2, KeSn3, Meh}




We concentrate on extending the results of Iwaniec, Luo, and Sarnak in \cite{ILS}. One of their key results is a formula for unweighted
sums of Fourier coefficients of holomorphic newforms of a given weight and level. This formula writes the unweighted sums in terms of weighted sums to which one can apply the Petersson trace formula; it is instrumental in performing any averaging over
holomorphic newforms, since one can interchange summation and replace the average of
Fourier coefficients with Kloosterman sums and Bessel functions, which are amenable to analysis.

A drawback of their formula is that it may only be applied
to averages of newforms of square-free level. One reason is that the development of
such a formula depends essentially on the construction of an explicit orthonormal
basis for the space of cusp forms of a given weight and level, which they only computed in the case of square-free level.
In 2011, Rouymi~\cite{Rouymi}
complemented the square-free calculations of Iwaniec, Luo, and Sarnak, finding an
orthonormal
basis for the space of cusp forms of prime power level, and applying this explicit
basis towards the development of a similar sum of Fourier coefficients over all
newforms with level equal to a fixed prime power.

In 2015, Blomer and Mili\'cevi\'c~\cite{BM} extended the results of Iwaniec, Luo,
and Sarnak and Rouymi by writing down an explicit orthonormal basis for the space of
cusp forms (holomorphic or Maass) of a fixed weight and, novelly, arbitrary level.

The purpose of this article is, first, to leverage the basis of
Blomer and Mili\'cevi\'c to prove an exact formula for sums of Fourier coefficients
of holomorphic newforms over all newforms of a given weight and level, where now the
level is permitted to be arbitrary (see below, as well as Proposition \ref{prop:peterssonondelta} for a detailed expansion). The basis of Blomer and Mili\'cevi\'c requires one to split over the square-free and square-full parts of the level; this splitting combined with the loss of several simplifying assumptions for Hecke eigenvalues and arithmetic functions makes the case where the level is not square-free much more complex. As an application, we use this formula to show the 1-level density agrees only with orthogonal symmetry.


\subsection{Harmonic averaging}\label{sec:harmave}

Throughout we assume that $k,N \geq 1$ with $k$ even. By $H_k^\star(N)$ we always mean a basis of arithmetically normalized Hecke eigenforms in the space orthogonal to oldforms. Explicitly, it is a basis of holomorphic cusp forms of weight $k$ and level $N$ which are new of level
$N$ in the sense of Atkin and Lehner~\cite{AL} and whose elements are eigenvalues of the Hecke operators $T_n$ with $(n,N) = 1$ and normalized so that the first Fourier coefficient is 1. We let $\lambda_f(n)$ denote the $n$\textsuperscript{th} Fourier coefficient of an $f \in H_k^\star(N)$ (see the next section for more details).

For any holomorphic cuspidal newform $f$, we introduce the renormalized Fourier coefficients
\begin{equation}\label{eq:psidef}
  \Psi_f(n) \ := \
  \left(\frac{\Gamma(k-1)}{(4\pi)^{k-1}}\right)^{1/2}||f||_N^{-1}\lambda_f(n),
\end{equation}
where $\|f\|_N^2=\langle f,f\rangle_N$ and $\langle \cdot, \cdot, \rangle_N$ denotes the Petersson inner product (see, for instance, \cite[(14.11)]{IK}).
We then define
\begin{equation}\label{deltapsi}
  \Delta_{k,N}(m,n) \ := \  \sum_{g \in \Bb_k(N)}\overbar{\Psi_g(m)}\Psi_g(n),
\end{equation}
where $\Bb_k(N)$ is an orthonormal basis for the space of cusp forms of weight $k$
and level $N$. The importance of $\Delta_{k,N}(m,n)$ is clarified by the
introduction of the Petersson formula in the next section.

Using the orthonormal basis $\Bb_k(N)$ of Mili\'cevi\'c and Blomer, we then prove
the following (unconditional) formula.
\begin{theorem}\label{thm:puresum}
Suppose that $(n,N)=1$. Then
\begin{equation}\label{eq:puresumeqn}
\sum_{f\in H_{k}^\star(N)}\lambda_f(n)
\ = \ \frac{k-1}{12}\sum_{LM=N}\mu(L)M\prod_{p^2\mid M}\lp\frac{p^2}{p^2-1}\rp^{-1}
\sum_{(m,M)=1}m^{-1}\Delta_{k,M}(m^2,n).
\end{equation}
\end{theorem}

A key part of the proof is a result on weighted sums of products of the Fourier coefficients, which we extract in Lemma \ref{DeltamnRelPrime}. Note that in many cases, the right-hand side of \eqref{eq:puresumeqn} is preferable to the left-hand side, as it is amenable to application of spectral summation formulas such as the Petersson formula (Proposition \ref{prop:petersson}) and can be studied via Kloosterman sums, see Proposition \ref{prop:peterssonondelta}. More generally, this sort of formula has a variety of applications involving the Fourier coefficients of holomorphic cusp forms and $L$-functions. Rouymi uses his basis and formula to study the non-vanishing at the central point of $L$-functions attached to primitive cusp
forms; we elect to apply our formula to generalize~\cite[Theorem 1.1]{ILS} on the
one-level density of families of holomorphic newform $L$-functions by removing the
condition that $N$ must pass to infinity through the square-free integers.

\subsection{The Density Conjecture}

Before stating our results, we introduce the $L$-function $L(s,f)$ associated to
a $f\in H_k^\star(N)$ as the Dirichlet series
\begin{equation} L(s,f)\ = \ \sum_1^\oo\lambda_f(n)n^{-s}. \end{equation}
See Section 3 of \cite{ILS} for the Euler product, analytic continuation, and functional equation
of $L(s,f)$; $L(s,f)$ may be analytically continued to an entire
function on $\C$ with a functional equation relating $s$ to $1-s$. We also need similar results for its symmetric square (see \cite{Delaunay, IK}): \begin{equation} L(s,\operatorname{sym}^2 f) \ = \ L(s, f \otimes f,) L(s, \chi)^{-1}\end{equation} where $\chi$ is the nebentypus of $f$ (we will only consider the case of trivial nebentypus below).


We now assume the Generalized Riemann Hypothesis for $L(s,f)$, and, for technical
reasons, $L(s,\operatorname{sym}^2 f)$ as well as for all Dirichlet $L$-functions (see Remark \ref{rek:GRHassumptions}). Then we may write all nontrivial zeros of
$L(s,f)$ as
\begin{equation} \varrho_f\ = \ \frac12+i\gamma_f. \end{equation}
For any $f\in H_k^\star(N)$, we denote by $c_f$ its analytic conductor; for our family
\begin{equation} c_f\ = \ k^2N. \end{equation}

Towards the definition of the one-level density for our families, we start with \eqref{one level density}, the one-level density for a fixed form $f$; the ordinates $\gamma_f$ are counted with their corresponding multiplicities,
and $\phi(x)$ is an even function of Schwartz class such that its Fourier
transform
\begin{equation} \widehat\phi(y)\ = \ \int_{-\oo}^\oo\phi(x)e^{-2\pi ixy}dx
\end{equation}
has compact support so that $\phi(x)$ extends to an entire function.

Our family $\mathscr F(N)$ is $H_k^\star(N)$, where the level $N$ is our
asymptotic parameter (and $\mathscr F=\cup_{N\geq1}\mathscr F(N)$).
It is worth mentioning that $\lim_{N\to\infty}|H_k^\star(N)| = \infty$. The one-level density is the
expectation of $D_1(f;\phi)$ averaged over our family:
\begin{equation}
  D_1(H_k^\star(N);\phi)\ := \ \frac1{|H_k^\star(N)|}\sum_{f\in H_k^\star(N)}D_1(f;\phi).
\end{equation}


Iwaniec, Luo, and Sarnak  \cite{ILS} prove the Density
Conjecture with the support of $\widehat\phi$ in $(-2,2)$ and as $N$ runs over
square-free numbers. We prove the following theorem with no conditions on how
$N$ tends to infinity; new features emerge from the presence of square factors dividing the level.

\begin{theorem}\label{thm:densitythm}
  Fix any $\phi\in \mathscr{S}(\R)$ with $\supp \widehat{\phi}\subset (-2,2)$. Then,
  assuming the Generalized Riemann Hypothesis for $L(s,f)$ and  $L(s,\operatorname{sym}^2 f)$ for $f \in H_k^\star(N)$ and for all Dirichlet $L$-functions,
  \begin{equation}
    \lim_{N\rightarrow \infty} \frac{1}{\abs{ H_k^\star(N)}}\sum_{f\in H_k^\star(N)}
    D_1(f;\phi)\ = \ \int_{-\infty}^\infty \phi(x)W_1(\mathrm{O})(x)\,dx
  \end{equation} where $W_1(\mathrm{O})(x)=1+\frac{1}{2}\delta_0(x)$; thus the 1-level density for the family $H_k^\star(N)$ agrees only with orthogonal symmetry.

  More generally, under the same assumptions the Density Conjecture holds for the family $H_k^\star(N)$ for any test function
  $\phi(x)$ whose Fourier transform is supported inside $(-u,u)$ with $u  <  2\log(kN) / \log(k^2 N)$.
\end{theorem}

\begin{remark} While \cite{ILS} are also able to split the family by the sign of the functional equation, we are unable to do so. The reason is that for square-free level $N$ the sign of the functional equation, $\epsilon_f$, is given by \begin{equation} \epsilon_f \ = \ i^k \mu(N) \lambda_f(N) N^{1/2} \end{equation} (see equation (3.5) of \cite{ILS}). By multiplying by $\frac12(1 \pm \epsilon_f)$ we can restrict to just the even ($\epsilon_f = 1$) or odd ($\epsilon_f = -1$) forms, at the cost of having an additional $\lambda_f(N)$ factor in the Petersson formula. This leads to involved calculations of Bessel-Kloosterman terms, but these sums can be evaluated well enough to obtain support in $(-2, 2)$. Unfortunately there is no analogue of their equation (3.5) for general level. \end{remark}


\begin{remark}\label{rek:GRHassumptions} We briefly comment on the use of the various Generalized Riemann Hypotheses. First, assuming GRH for $L(s,f)$ yields a nice spectral interpretation of the 1-level density, as the zeros now lie on a line and it makes sense to order them; note, however, that this statistic is well-defined even if GRH fails. Second, GRH for $L(s,\operatorname{sym}^2 f)$ is used to bound certain sums which arise as lower order terms; in \cite{ILS} (page 80 and especially page 88) the authors remark how this may be replaced by additional applications of the Petersson formula (assuming GRH allows us to trivially estimate contributions from each form, but a bound on average suffices). Finally, GRH for Dirichlet $L$-functions is needed when we follow \cite{ILS} and expand the Kloosterman sums in the Petersson formula with Dirichlet characters; if we do not assume GRH here we are still able to prove the 1-level density agrees with orthogonal symmetry, but in a more restricted range.
\end{remark}

\ \\
The structure of the paper is as follows. Our main goal is to prove the formula for sums of Hecke eigenvalues and then use this to compute the one-level density. We begin in \S\ref{sec:preliminaries} with a short introduction of the theory of primitive holomorphic cusp forms, as well as the Petersson trace formula and the basis of Blomer and Mili\'cevi\'c. In \S\ref{sec:deltaform} we find a formula for $\Delta_{k,N}(m,n)$, which we leverage in \S\ref{sec:weightedunweighted} to find a formula for the arithmetically weighted sums, $\Delta^\star_{k,N}(n)$ (see \cite[(2.53)]{ILS}); this is Theorem \ref{thm:puresum}. Using our formula, we find bounds for $\Delta^\star_{k,N}(n)$ in \S\ref{sec:pure}, culminating in the computation of the one-level density in \S\ref{sec:densityconj} (Theorem \ref{thm:densitythm}).


\section{Preliminaries}\label{sec:preliminaries}
In this section we introduce some notation and results to be used throughout, much of which can be found in \cite{IK}.


\subsection{Hecke eigenvalues and the Petersson inner product}

Our setup is classical. Throughout $k, N$ are positive integers, with $k$ even.
Let $S_k(N)$ be the linear space spanned by cusp forms of weight $k$ and trivial
nebentypus which are Hecke eigenforms for the congruence group $\Gamma_0(N)$.
Each $f \in S_k(N)$ admits a Fourier development
\begin{equation}
f(z) \ = \  \sum_{n\geq 1}a_f(n)e(nz),
\end{equation}
where $e(z):=e^{2\pi i z}$ and the $a_f(n)$ are in general complex numbers,
though as we only consider forms with trivial nebentypus, our Fourier coefficients
are real.

It is well known that $S_k(N)$ is a finite-dimensional Hilbert space with respect to the Petersson inner
product
\begin{equation}
\left<f,g\right>_N \ = \  \int_{\Gamma_0(N)\backslash\Hp}f(z)\overline{g(z)}y^{k-2}dxdy,
\end{equation}
where $\Hp$ denotes the upper-half plane $\Hp=\{z\in\mathbb{C} : \Im(z)>0\}$.
Given a form on $\Gamma_0(M)$, it is possible to induce a form on $\Gamma_0(N)$ for
$M\mid N$. We call such forms for which $M<N$ ``old forms''; the basis of ones orthogonal to the space spanned by the forms with $M < N$
which are eigenvalues of the Hecke operators are called the ``new forms'' or ``primitive forms''.  We write the inner product with a subscript $N$ to indicate we are considering $f$ and $g$ as forms on $\Gamma_0(N)$, when perhaps $\langle f,g\rangle_M$ might make sense as well.


Atkin and Lehner \cite{AL} showed that the space $S_k(N)$ has the following canonical
orthogonal decomposition in terms of the newforms described in \S\ref{sec:harmave}: let $H_k^\star(M)$ be a basis of arithmetically normalized Hecke eigenformsforms for the space of newforms of weight $k$ and level $M$ (typically we choose $M$ to be a divisor of $N$). Then
\begin{equation}
S_k(N) \ = \  \bigoplus_{LM=N}\bigoplus_{f\in H_k^\star(M)}S_k(L;f)
\end{equation}
where $S_k(L;f)$ is the linear space spanned by the forms
\begin{equation}
f_{\mid \ell}(z)\ = \ \ell^{\frac{k}{2}}f(\ell z)\quad\text{with } \ell\mid L.
\end{equation}
Though the forms $f_{\mid\ell}(z)$ are linearly independent, they are not
orthogonal.

If $f \in H_k^\star(M)$ then $f$ is an eigenfunction of all Hecke operators
$T_M(n)$, where
\begin{equation}
(T_M(n)f)(z) \ = \  \frac{1}{\sqrt{n}}\sum_{\substack{ad=n\\(a,M)=1}}\lp\frac{a}{d}\rp^{k/2}\sum_{b\pmod d}f\lp\frac{az+b}{d}\rp.
\end{equation}
For a fixed $f\in H_k^\star(M)$, let $\lambda_f(n)$ denote the eigenvalue of
$T_M(n)$; i.e.,
\begin{equation}
T_M(n)f \ = \  \lambda_f(n)f
\end{equation}
for all $n\geq 1$. The Hecke eigenvalues are multiplicative; more precisely, they
satisfy the following identity for any $m,n\geq 1$:
\begin{equation}\label{heckemultiplicativity}
\lambda_f(m)\lambda_f(n)\ = \ \sum_{\substack{d\mid(m,n) \\ (d,M)=1}}\lambda_f(mn/d^2).
\end{equation}
We normalize so that
\begin{equation}
a_f(1) \ = \  1.
\end{equation}
Then $a_f(n)$ and $\lambda_f(n)$ are related by
\begin{equation}
a_f(n) \ = \  \lambda_f(n)n^{(k-1)/2}.
\end{equation}
Deligne showed that the Weil conjectures imply the Ramanujan-Petersson conjecture
for holomorphic cusp forms, and then proved them. As a consequence, for
$f\in H_k^\star(N)$ we have the bound
\begin{equation}
  \left|\lambda_f(n)\right|\ \leq \ \tau(n),
\end{equation}
where $\tau(n)$ is the divisor function, and if $f\in H_k^\star(M)$ and $p\mid M$,
then
\begin{equation}\label{deligne}
  \lambda_f(p)^2 \ = \  \begin{cases}
    \frac{1}{p} &\text{if }p \mid\mid M\\
    0 &\text{if }p^2\mid M,
  \end{cases}
\end{equation}
see, for instance, \cite[(2.8)]{Rouymi}.

We recall the definition~\eqref{eq:psidef} of the normalized Fourier coefficients
$\Psi_f(n)$ attached to any cusp form $f$:
\begin{equation}
\Psi_f(n) \ = \  \left(\frac{\Gamma(k-1)}{(4\pi n)^{k-1}}\right)^{1/2}||f||_N^{-1}a_f(n)
\ \ll_f\ \tau(n).
\end{equation}
Let $\Bb_k(N)$ be an orthogonal basis of $S_k(N)$. Then
\begin{equation}
  \left|\Bb_k(N)\right|\ = \ \dim S_k(N)\ \asymp\ \nu(N)k
\end{equation}
where
\begin{equation}
  \nu(N)\ := \ [\Gamma_0(1):\Gamma_0(N)]\ = \ N\prod_{p\mid N}(1+\tfrac1p).
\end{equation}
From the Atkin-Lehner decomposition, we also deduce
\begin{equation}
  \dim S_k(N)\ = \ \sum_{LM=N}\tau(L)\left|H_k^\star(M)\right|.
\end{equation}
Recall Definition~\eqref{deltapsi} of $\Delta_{k,N}(m,n)$:
\begin{equation}
  \Delta_{k,N}(m,n) \ := \  \sum_{g \in \Bb_k(N)}\overbar{\Psi_g(m)}\Psi_g(n).
\end{equation}
The importance of $\Delta_{k,N}(m,n)$ is established by the Petersson trace formula.
\begin{proposition}[{Petersson~\cite{Petersson}}]\label{prop:petersson}
  For any $m,n\geq1$ we have
  \begin{equation}
    \Delta_{k,N}(m,n)\ = \ \delta(m,n)+2\pi i^k\sum_{c\equiv0\pmod N}c^{-1}S(m,n;c)
    J_{k-1}\left(\frac{4\pi\sqrt{mn}}c\right).
  \end{equation}
\end{proposition}

Though the quantity $\Delta_{k,N}(m,n)$ is independent of the choice of an orthonormal basis, we would like to
compute with the Petersson trace formula using an explicit basis $\Bb_k(N)$ to
average over newforms. However, as remarked, the spaces $S_k(L;f)$ do not have a
distinguished orthogonal basis. Therefore, to produce a basis $\Bb_k(N)$, we need
a basis for the spaces $S_k(L;f)$. Iwaniec, Luo, and Sarnak~\cite{ILS}
write down an explicit basis when $N$ is square-free. As we will see in the next
section, Blomer and Mili\'cevi\'c~\cite{BM} have recently obtained a basis for
arbitrary level $N$. Our first key idea, a kind of trace formula for sums of
Hecke eigenvalues over newforms in the case $N$ is arbitrary, is an explicit
computation with this new basis. Our second key idea on the one-level density of
the $L$-functions $L(s,f)$ for $f\in H_k^\star(N)$ uses our first key idea in an
essential way to reduce the problem to the one already treated by Iwaniec, Luo, and
Sarnak.

To $H_k^\star(M)$ we often associate $\chi_{0;M}$,
the trivial character $\bmod$ $M$:
\begin{equation}
\chi_{0;M}(n) \ = \  \begin{cases}
			 1 & \text{if } (n,M)=1\\
			 0 & \text{otherwise.}
		  \end{cases}
\end{equation}

\subsection{An orthonormal basis for $S_k(N)$}

For $f \in H_k^\star(M)$ consider the following arithmetic functions, which coincide with the ones defined in~\cite{BM}
up to a few corrections~\cite{BM2}:
\begin{equation}
r_f(c) \ := \ \sum_{b\mid c}\frac{\mu(b)\lambda_f(b)^2}{b\sigma_{-1}^{\rm twisted}(b)^2},\,\, \alpha(c)\ := \ \sum_{b\mid c}\frac{\chi_{0;M}(b)\mu(b)}{b^2},\,\, \beta(c)\ := \ \sum_{b\mid c}\frac{\chi_{0;M}(b)\mu^2(b)}{b},
\end{equation}
where $\mu_f(c)$ is the multiplicative function given implicitly by
\begin{equation}
L(f,s)^{-1} \ = \  \sum_{c}\frac{\mu_f(c)}{c^s},
\end{equation}
or explicitly on prime powers by
\begin{equation}
  \mu_f(p^j)\ = \ \begin{cases}-\lambda_f(p)&j=1 \\ \chi_{0;M}(p)&j=2 \\ 0 & j>2 \end{cases}
\end{equation}
and
\begin{equation}
  \sigma_{-1}^{\rm twisted}(b) \ = \  \sum_{r\mid b}\frac{\chi_{0;M}(r)}{r}.
\end{equation}

For $\ell\mid d$ define
\begin{equation}
\xi'_d(\ell) \ := \  \frac{\mu(d/\ell)\lambda_f(d/\ell)}{r_f(d)^{1/2}(d/\ell)^{1/2}\beta(d/\ell)},\,\, \ \  \xi''_d(\ell) \ := \  \frac{\mu_f(d/\ell)}{(d/\ell)^{1/2}(r_f(d)\alpha(d))^{1/2}}.
\end{equation}
Write $d = d_1d_2$ where $d_1$ is square-free, $d_2$ is square-full, and $(d_1,d_2) = 1$. Thus $p\mid\mid d$ implies $p\mid d_1$ and $p^2\mid d$ implies $p^2\mid d_2$. Then for $\ell \mid d$ define
\begin{equation}
\xi_d(\ell)\ := \  \xi'_{d_1}((d_1, \ell))\xi''_{d_2}((d_2,\ell)).
\end{equation}

We record the following identities; while these are not needed for the arguments in this paper, they were of use in an earlier draft in investigating related problems, and thus may be of use to others. For prime powers, $\xi_d(\ell)$ simplifies as
\begin{equation}\label{eq:identitiesforprooflemmaxisimplify}\begin{aligned}[b]
  \xi_1(1) &\ = \  1,
  & \xi_{p^\nu}(p^\nu) &\ = \  \pez{r_f(p)\pez{1-\chi_{0;M}(p)/p^2}}^{-1/2} \\
  \xi_p(p) &\ = \  r_f(p)^{-1/2},
  & \xi_{p^\nu}(p^{\nu-1}) &\ = \  \tfrac{-\lambda_f(p)}{\sqrt{p}}\xi_{p^\nu}(p^\nu) \\
  \xi_p(1) &\ = \  \tfrac{-\lambda_f(p)}{\sqrt{p}\pez{1 + \chi_{0;M}(p)/p}}\xi_p(p),
  & \xi_{p^\nu}(p^{\nu - 2})&\ = \ \tfrac{\chi_{0;M}(p)}{p}\xi_{p^\nu}(p^\nu), \text{ for } \nu \geq 2.
\end{aligned}\end{equation}

Blomer and Mili\'cevi\'c prove the following.
\begin{proposition}[Blomer and Mili\'cevi\'c {\cite[Lemma 9]{BM}}]\label{ONB}
Let
\begin{equation}
f_d(z)\ := \ \sum_{\ell\mid d}\xi_d(\ell)f\mid_ \ell(z),
\end{equation}
 where $N = LM$ and $f \in H_k^\star(M)$ is Petersson-normalized with respect to the Petersson norm on $S_k(N)$. Then $\{f_d : d\mid L\}$ is an orthonormal basis of $S_k(L;f)$.
\end{proposition}

Note that in our application we are not using the Petersson normalization but instead have normalized our forms to have first coefficient 1; thus for us below we have an orthogonal basis which becomes orthonormal upon dividing the forms by their norm.


In addition, we will also use of the following lemma. Originally stated in the
context of square-free level, the same proof holds in general.
\begin{lemma}[Iwaniec, Luo, Sarnak {\cite[Lemma 2.5]{ILS}}] Let $M|N$ and $f\in H_k^\star(M)$. Then
\begin{equation}\label{ILSLemma2.5}
\left<f,f\right>_N \ = \  (4\pi)^{1-k}\Gamma(k)\frac{\nu(N)\varphi(M)}{12M}Z(1,f),
\end{equation}
where $\varphi$ is the Euler totient function, and
\begin{equation}
Z(s,f) \ := \  \sum_{n=1}^\infty\lambda_f(n^2)n^{-s} \ \ {\rm with}\ \ {\rm Res}_{s=1} L(s, f\otimes f) \ = \ Z(1,f) M / \varphi(M).
\end{equation}
\end{lemma}

It is often convenient to work with the local zeta function
\begin{equation}
  Z_N(s,f)\ := \ \sum_{\ell\mid N^\oo}\lambda_f(\ell^2)\ell^{-s}.
\end{equation}
If $f \in H_k^\star(N)$, then we may write $Z(1,f) = \prod_{p}Z(1,f)$ and one deduces from~\eqref{deligne} and ~\eqref{heckemultiplicativity} that the local Euler
factors $Z_p(1,f)$ are given by
\begin{align}
Z_p(1,f) \ = \  \begin{cases}
\left(1+\frac{1}{p}\right)^{-1}\rho_f(p)^{-1} & \text{if }p\nmid N \\
\left(1+\frac{1}{p}\right)^{-1}\left(1-\frac{1}{p}\right)^{-1}&\text{if }p\mid\mid N \\
1 & \text{if } p^2\mid N,
\end{cases}
\end{align}
where $\rho_f(c)$ is the multiplicative function
\begin{equation}\label{eq:rhofdef}
\rho_f(c) \ = \  \sum_{b\mid c}\mu(b)b\lp\frac{\lambda_f(b)}{\nu(b)}\rp^2 \ = \  \prod_{p\mid c}\lp1 - p\lp\frac{\lambda_f(p)}{p+1}\rp^2\rp.
\end{equation}
Assume now that $N=LM$ and $f\in H_k^\star(M)$ is normalized so that $\lambda_f(n) = 1$, and, writing
\begin{equation}
\mathfrak{p}(L,M) \ := \  \prod_{\substack{p^{\beta}\mid\mid L\\ p\mid M}}p^\beta,
\end{equation}
note that
\begin{equation}
Z_{LM/\mathfrak{p}(L,M)}(s,f) \ = \  Z_N(s,f).
\end{equation}
Thus, without loss of generality, we assume for the next calculation that $(L,M) = 1$; in particular, no prime divides both $L$ and $M$. Specializing to $s=1$, we find
\begin{align*}
\frac{MN}{\phi(M)\nu(N)}\prod_{\substack{p\mid L\\ p\nmid M}}\rho_f(p)^{-1}\prod_{p^2\mid M}\lp\frac{p^2-1}{p^2}\rp &\ = \  \prod_{p^2\mid M}\lp 1-\frac{1}{p}\rp^{-1}\lp 1 + \frac{1}{p}\rp^{-1}\prod_{p \mid \mid M}\lp 1-\frac{1}{p}\rp^{-1}\lp 1 + \frac{1}{p}\rp^{-1}
\\& \qquad \times\prod_{\substack{p\mid N\\ p\nmid M}}\lp 1 + \frac{1}{p}\rp^{-1}\prod_{\substack{p\mid L\\ p\nmid M}}\rho_f(p)^{-1}\prod_{p^2\mid M}\lp\frac{p^2-1}{p^2}\rp
\\&\ = \  \prod_{p\mid\mid M}\lp 1 - \frac{1}{p}\rp^{-1}\lp 1 + \frac{1}{p}\rp^{-1}\prod_{p\nmid M \atop p|N}\lp1 + \frac{1}{p}\rp^{-1}\rho_f(p)^{-1}
\\&\ = \  \prod_{p\mid N}Z_p(1,f),
\eeqno
\end{align*}
since $f\in H_k^\star(M)$. We obtain
\begin{equation}\label{eq:Zfactorization}
  \frac{MN}{\varphi(M)\nu(N)}\prod_{\substack{p\mid L \\
  p\nmid M}}\rho_f(p)^{-1} \ = \  Z_{N}(1,f)\prod_{p^2\mid M}\lp\frac{p^2}{p^2-1}\rp.
\end{equation}
We also note that if $p\nmid M$, then $r_f(p)=\rho_f(p)$, and if $p \mid M$, then
$r_f(p) = 1 - \lambda_f(p)^2/p$.


\section{A Formula for $\Delta_{k,N}(m,n)$}\label{sec:deltaform}
In this section we provide the following explicit formula for $\Delta_{k,N}(m,n)$ in terms of Hecke eigenvalues. We start with a generalization of Lemma 2.7 of \cite{ILS} to general $N$.

\begin{lemma}\label{DeltamnRelPrime}
Suppose $(m,N)=1$ and $(n,N)=1$. Then
\begin{equation}
\Delta_{k,N}(m,n) \ = \ \frac{12}{(k-1)N} \prod_{p^2\mid N}\lp\frac{p^2}{p^2-1}\rp\sum_{LM=N} \sum_{f\in H^\star_k(M)}\frac{Z_N(1,f)}{Z(1,f)}\lambda_f(m)\lambda_f(n).
\end{equation}
\end{lemma}

Before we can prove the above lemma, we begin with a result about the coefficients inherited from the orthonormal basis defined in Proposition \ref{ONB}. Note that if $f(z)\in  H_k^\star(M)$ has Fourier expansion
\begin{equation}
f(z) \ = \  \sum_{n\geq 1}a_f(n)e(nz),
\end{equation}
then
\begin{align}
f_d(z)& \ := \  \sum_{\ell\mid d}\xi_d(\ell)f\mid_ \ell(z) \ = \  \sum_{\ell\mid d}\xi_d(\ell)\ell^{k/2}f(\ell z),
\end{align}
so the coefficients of the Fourier expansion of $f_d(z)$ are given by
\begin{equation}
a_{f_d}(n) \ = \  \sum_{\ell\mid(d,n)}\xi_d(\ell)\ell^{k/2}a_f(\tfrac{n}{\ell}).
\end{equation}

Let $N=LM$ and let $f$ be a newform of weight $k$ and level $M$.
Let $f' = f/||f||_N$ so that $f'$ is Petersson-normalized with respect to level $N$ (i.e., $||f'||_N = 1$) and note that
\begin{equation}
a_{f'}(n) \ = \  \frac{\lambda_f(n)n^{(k-1)/2}}{||f||_N}.
\end{equation}
Then by Proposition \ref{ONB}, the set $\{f'_d: d\mid L\}$ is an orthonormal basis
of $S_k(L;f)$. Let
\begin{equation}
\Bb_k(N) \ = \  \bigcup_{LM = N}\bigcup_{f\in H_k^\star(M)}\bigcup_{d\mid L}\{f'_d\}
\end{equation}
be our orthonormal basis for $S_k(N)$. We have
\begin{equation}\label{eq:deltacoeffs}
  \begin{aligned}[b]  
    \Delta_{k,N}(m,n) \ = &\ \sum_{g\in\Bb_k(N)}\overline{\left(\frac{\Gamma(k-1)}{(4\pi m)^{k-1}}\right)^{1/2}||g||_N^{-1}a_g(m)}\left(\frac{\Gamma(k-1)}{(4\pi n)^{k-1}}\right)^{1/2}||g||_N^{-1}a_g(n)
\\& = (4\pi)^{1-k}(mn)^{\tfrac{1-k}{2}}\Gamma(k-1)\sum_{LM=N}\sum_{f \in H_k^\star(M)}\sum_{f'_d:d\mid L}||f'_d||_N^{-2}\overline{a_{f'_d}(m)}a_{f'_d}(n)\\
     \ &\ = (4\pi)^{1-k}(mn)^{\tfrac{1-k}{2}}\Gamma(k-1) \\& \ \ \ \ \ \times \sum_{LM=N}\sum_{f \in H_k^\star(M)}\sum_{d\mid L}\frac{1}{||f||_N^2}\overline{\left(\sum_{\ell \mid (d,m)}\xi_d(\ell)\ell^{k/2}\lambda_f(\tfrac{m}{\ell})(m/\ell)^{(k-1)/2}\right)}\\& \ \ \ \ \ \times\left(\sum_{\ell \mid (d,n)}\xi_d(\ell)\ell^{k/2}a_f(\tfrac{n}{\ell})(n/\ell)^{(k-1)/2}\right)\\
    \ &\ = \frac{12}{(k-1)\nu(N)}\sum_{LM=N}\frac{M}{\varphi(M)}\sum_{f\in H_k^\star(M)}\frac{1}{Z(1,f)} \\
    &\ \qquad \qquad \times \sum_{d\mid L}\left(\sum_{\ell \mid (d,m)}\xi_d(\ell)\ell^{1/2}\lambda_f(\tfrac{m}{\ell})\right)\left(\sum_{\ell \mid (d,n)}\xi_d(\ell)\ell^{1/2}\lambda_f(\tfrac{n}{\ell})\right),
  \end{aligned}
\end{equation}
where the last equality follows from (\ref{ILSLemma2.5}).

We now specialize to $(n,N)=1$ and $(m,N)=1$, as these assumptions simplify the calculations significantly and are sufficient for our application to the one-level density. In particular, as $d|L|N$ and $(m,N) = (n,N) = 1$, $\ell|(d,m)$ implies $\ell = 1$ (and similarly for $\ell|(d,n)$). Thus the previous equation simplifies to
\begin{equation}\label{eq:simpleDelta}
\Delta_{k,N}(m,n) \ = \ \frac{12}{(k-1)\nu(N)}\sum_{LM=N}\frac{M}{\varphi(M)}\sum_{f\in H_k^\star(M)}\frac{\lambda_f(m)\lambda_f(n)}{Z(1,f)}\sum_{d\mid L}\xi_d(1)^2.
\end{equation}
The task is now to understand $$\sum_{d\mid L}\xi_d(1)^2,$$
which we investigate in the following lemma.

\begin{lemma}\label{xisimplify} Let $f$ be as before, with $LM = N$. We have
\begin{equation}
\sum_{d\mid L}\xi_d(1)^2\ =\ \prod_{\substack{p\mid L\\p\nmid M}} \rho_f(p)^{-1}\prod_{\substack{p^2\mid N\\p^2\nmid M}} \frac{p^2}{p^2-1}.
\end{equation}
\end{lemma}

\begin{proof} Using the definition of $\xi_d(\ell)$, writing $d=d_1d_2$, where $d_1$ is squarefree and $d_2$ is squarefull and $(d_1,d_2)=1$, we have
\begin{equation}\label{squaressum} \sum_{d\mid L} \xi_d(1)^2 \ = \ \sum_{d\mid L} \frac{\lambda_f(d_1)^2\mu_f(d_2)^2}{r_f(d)d\beta(d_1)^2\alpha(d_2)}.  \end{equation} Recall that $\mu_f(1) = 1, \mu_f(p^2)=\chi_{0;M}(p)$, and $ \mu_f(p^{e_p})=0$ for all $e_p>2$. As all functions in the sum above are multiplicative, we can factor as follows:
\begin{equation}
\sum_{d\mid L} \xi_d(1)^2 \ = \  \prod_{p^{e_p}\| L} \left( 1 + \frac{\lambda_f(p)^2}{r_f(p)p\beta(p)^2} + \frac{\chi_{0;M}(p)(1-\mu(p^{e_p})^2)}{r_f(p^2)p^2\alpha(p^2)}\right). \end{equation} We now break into cases when $(p,M)=1$ and $(p,M)\neq 1$ to remove the $\chi_{0;M}(p)$: \begin{equation}\label{xifact}
\sum_{d\mid L} \xi_d(1)^2 \ = \ \prod_{\substack{p^{e_p}\| L\\(p,M)=1}} \left( 1 + \frac{\lambda_f(p)^2}{r_f(p)p\beta(p)^2} + \frac{1-\mu(p^{e_p})^2}{r_f(p^2)p^2\alpha(p^2)}\right)\prod_{\substack{p^{e_p}\| L\\(p,M)\neq 1}} \left( 1 + \frac{\lambda_f(p)^2}{r_f(p)p\beta(p)^2}\right).
\end{equation}
We can now simplify many of the terms as follows. If $(p,M)=1$, then
\begin{align*}
\beta(p)^2\ = &\ (1+1/p)^2 \\
\alpha(p^2)\ = &\ (1-1/p^2)\\
r_f(p)\ = &\ \rho_f(p).\eeqno
\end{align*} If $(p,M)\neq 1$, we have
\begin{align*}
\beta(p)\ =&\  \alpha(p^2)=1\\
 r_f(p)\ =&\  1-\frac{\lambda_f(p)^2}{p}.\eeqno
\end{align*}
In addition, note that $r_f(p)=r_f(p^2)$. Thus we can write the right hand side of \eqref{xifact} as
\begin{align}\label{eq:xifactjustrightrewritten}
&\ \prod_{\substack{p^{e_p}\| L\\(p,M)=1}} \left( 1 + \frac{\lambda_f(p)^2}{\rho_f(p)p\left(1+\frac{1}{p}\right)^2} + \frac{1-\mu(p^{e_p})^2}{\rho_f(p)p^2\left(1-\frac{1}{p^2}\right)}\right)\prod_{\substack{p^{e_p}\| L\\(p,M)\neq 1}} \left( 1 + \frac{\lambda_f(p)^2}{p\left(1-\frac{\lambda_f(p)^2}{p}\right)}\right).
\end{align}
Recall that because $f\in H_k^\star(M)$,
\begin{equation}
\lambda_f(p)^2 = \begin{cases}
\frac{1}{p} &\text{if }p\| M\\
0 &\text{if }p^2\mid M.
\end{cases}
\end{equation}
Therefore we can rewrite the second product in \eqref{eq:xifactjustrightrewritten}, and obtain
\begin{align}
\sum_{d\mid L} \xi_d(1)^2 \ = &\ \prod_{\substack{p^{e_p}\| L\\(p,M)=1}} \left( 1 + \frac{\lambda_f(p)^2}{\rho_f(p)p\left(1+\frac{1}{p}\right)^2} + \frac{1-\mu(p^{e_p})^2}{\rho_f(p)p^2\left(1-\frac{1}{p^2}\right)}\right)\prod_{\substack{p\mid L\\p\| M}} \left( \frac{p^2}{p^2-1}\right).
\end{align}
We now simplify the first product above as
\begin{align*}
&\ \prod_{\substack{p^{e_p}\| L\\(p,M)=1}} \left( \frac{\rho_f(p)p\left(1+\frac{1}{p}\right)^2(p^2-1) +\lambda_f(p)^2(p^2-1) + p\left(1+\frac{1}{p}\right)^2(1-\mu(p^{e_p})^2)}{\rho_f(p)p\left(1+\frac{1}{p}\right)^2(p^2-1)}\right)\nonumber\\
&\qquad\times \prod_{\substack{p\mid L\\p\| M}} \left( \frac{p^2}{p^2-1}\right)\\
=& \  \prod_{\substack{p^{e_p}\| L\\(p,M)=1}} \left( \frac{\left(p\left(1+\frac{1}{p}\right)^2-\lambda_f(p)^2\right)(p^2-1) +\lambda_f(p)^2(p^2-1) + p\left(1+\frac{1}{p}\right)^2(1-\mu(p^{e_p})^2)}{\rho_f(p)p\left(1+\frac{1}{p}\right)^2(p^2-1)}\right)\nonumber\\
&\qquad\times\prod_{\substack{p\mid L\\p\| M}} \left( \frac{p^2}{p^2-1}\right)\\
=&\ \prod_{\substack{p^{e_p}\| L\\(p,M)=1}} \left( \frac{p\left(1+\frac{1}{p}\right)^2(p^2-\mu(p^{e_p})^2)}{\rho_f(p)p\left(1+\frac{1}{p}\right)^2(p^2-1)}\right)\prod_{\substack{p\mid L\\p\| M}} \left( \frac{p^2}{p^2-1}\right)
\\
=&\prod_{\substack{p\mid L\\p\nmid M}} \rho_f(p)^{-1}\prod_{\substack{p^2\mid N\\p^2\nmid M}} \left( \frac{p^2}{p^2-1}\right),\eeqno
\end{align*}
which completes the proof.
\end{proof}

Combining Lemma~\ref{xisimplify} with equations \eqref{eq:Zfactorization} and \eqref{eq:simpleDelta} yields Lemma~\ref{DeltamnRelPrime}.



\section{An inversion and a change from weighted to pure sums}\label{sec:weightedunweighted}

We now introduce the arithmetically weighted sums, as defined in \cite[(2.53)]{ILS},
\begin{align}
\Delta_{k,N}^*(m,n) \ = \  \sum_{f\in H_k^\star(N)}\frac{\lambda_f(n)\lambda_f(m)Z_N(1,f)}{Z(1,f)}.
\end{align}
This allows us to state one of our main results, which generalizes work of Iwaniec, Luo,
and Sarnak~\cite[Proposition 2.8]{ILS} and Rouymi~\cite[Proposition 2.3]{Rouymi}.
\begin{proposition}\label{prop:DeltaDelta}
Suppose $(m,N)=1$ and $(n,N)=1$. Then
\begin{equation}\label{eq:preinversion}
\Delta_{k,N}(m,n) \ = \  \frac{12}{(k-1)N}\prod_{p^2\mid N}\lp\frac{p^2}{p^2-1}\rp\sum_{LM=N}\sum_{\substack{\ell \mid L^\infty\\ (\ell,M)=1}}\ell^{-1}\Delta_{k,M}^\star(m\ell^2,n)
\end{equation}
and
\begin{equation}\label{eq:postinversion}
\Delta_{k,N}^\star(m,n) \ = \  \frac{k-1}{12}\sum_{LM=N}\mu(L)M\prod_{p^2\mid M}\lp\frac{p^2}{p^2-1}\rp^{-1}\sum_{\substack{\ell\mid L^\infty\\ (\ell,M)=1}}\ell^{-1}\Delta_{k,M}(m\ell^2,n).
\end{equation}
\end{proposition}
\begin{proof}[Proof of Proposition~\ref{prop:DeltaDelta}]
We first prove \eqref{eq:preinversion}.
Note the following: $(m,N)=1$, $(n,N) = 1$ and $\ell\mid L^\infty$ imply $(m,M) = 1$, $(n,M)=1$, and $(\ell,m)=1$.

These observations together with Lemma \ref{DeltamnRelPrime} imply
\begin{align*}
\Delta_{k,N}(m,n) &\ = \  \frac{12}{(k-1)N}\prod_{p^2\mid N}\lp\frac{p^2}{p^2-1}\rp\sum_{LM=N} \sum_{f\in H^\star_k(M)}\frac{Z_{L/\mathfrak{p}(L,M)}(1,f)Z_M(1,f)}{Z(1,f)}\lambda_f(m)\lambda_f(n)
\\& \ = \  \frac{12}{(k-1)N}\prod_{p^2\mid N}\lp\frac{p^2}{p^2-1}\rp\sum_{LM=N}\sum_{f\in H_k^\star(M)}\Big(\sum_{\substack{\ell \mid L^\infty \\(\ell,M)=1}}\lambda_f(\ell^2)\ell^{-1}\Big)\frac{Z_M(1,f)}{Z(1,f)}\lambda_f(m)\lambda_f(n)
\\& \ = \  \frac{12}{(k-1)N}\prod_{p^2\mid N}\lp\frac{p^2}{p^2-1}\rp\sum_{LM=N}\sum_{\substack{\ell \mid L^\infty\\ (\ell,M)=1}}\ell^{-1}\Delta_{k,M}^\star(m\ell^2,n). \eeqno
\end{align*}

We are now ready to prove \eqref{eq:postinversion} using M\"obius inversion.
We begin with
\begin{align*}
&\frac{k-1}{12}\sum_{LM=N}\mu(L)M\prod_{p^2\mid M}\lp\frac{p^2}{p^2-1}\rp^{-1}\sum_{\substack{\ell\mid L^\infty\\ (\ell,M)=1}}\ell^{-1}\Delta_{k,M}(m\ell^2,n) \\
& \ = \  \frac{k-1}{12}\sum_{LM=N}\mu(L)M\prod_{p^2\mid M}\lp\frac{p^2}{p^2-1}\rp^{-1}\sum_{\substack{\ell\mid L^\infty\\ (\ell,M)=1}}\ell^{-1}\frac{12}{(k-1)M}\prod_{p^2\mid M}\lp\frac{p^2}{p^2-1}\rp \\
& \qquad\qquad \times\sum_{QW=M}\sum_{f\in H_k^\star(W)}\frac{Z_M(1,f)}{Z(1,f)}\lambda_f(m\ell^2)\lambda_f(n) \\
& \ = \  \sum_{LM=N}\mu(L)\sum_{QW=M}\sum_{f\in H_k^\star(W)}
  \Big(\sum_{\substack{\ell\mid L^\infty\\(\ell,M)=1}}\lambda_f(\ell^2)\ell^{-1}\Big)
  \frac{Z_M(1,f)}{Z(1,f)}\lambda_f(m)\lambda_f(n) \\
& \ = \  \sum_{LM=N}\mu(L)\sum_{QW=M}\sum_{f\in H_k^\star(W)}\frac{Z_N(1,f)}{Z(1,f)}\lambda_f(m)\lambda_f(n).\eeqno
\end{align*}
Let
\begin{equation}
\mathcal{Z}_N(W)\ := \  \sum_{f\in  H_k^\star(W)}\frac{Z_N(1,f)}{Z(1,f)}\lambda_f(m)\lambda_f(n).
\end{equation}
Interchanging orders of summation yields
\begin{equation}
\sum_{LM=N}\mu(L)\sum_{QW=M}\mathcal{Z}_N(W)
  \ = \  \sum_{W\mid N}\mathcal{Z}_N(W)\sum_{L\mid \tfrac{N}{W}}\mu(L)\\
  \ = \   \Delta_{k,N}^\star(m,n),
\end{equation}
as the M\"obius sum vanishes unless $W=N$ and $\mathcal{Z}_N(N) = \Delta_{k,N}^\star(m,n)$.
\end{proof}
One of our primary applications of Proposition~\ref{prop:DeltaDelta} is to
obtain a formula for pure sums of Hecke eigenvalues. We define the pure sum
\begin{equation}
\Delta^\star_{k,N}(n)\ := \ \sum_{f\in H_{k}^\star(N)}\lambda_f(n)
\end{equation}
(not to be confused with $\Delta_{k,N}^\star(m,n)$, which is weighted)
and prove Theorem \ref{thm:puresum} from the introduction, which we restate here for convenience.

\begin{reptheorem}{thm:puresum}
Suppose that $(n,N)=1$. Then
\begin{equation}
\Delta^\star_{k,N}(n)\ = \ \frac{k-1}{12}\sum_{LM=N}\mu(L)M\prod_{p^2\mid M}\lp\frac{p^2}{p^2-1}\rp^{-1}\sum_{(m,M)=1}m^{-1}\Delta_{k,M}(m^2,n).
\end{equation}
\end{reptheorem}

\begin{proof}
We remove the weights in \eqref{prop:DeltaDelta} by summing $m^{-1}\Delta^\star_{k,N}(m^2,n)$ over all $(m,N)=1$. We will need to replace $\sum_{\ell\mid N^\infty} \sum_{(m,N)=1}(\ell m)^{-1}\lambda_f(\ell^2)\lambda_f(m^2)$ with $\sum_{r\ge 1}r^{-1}\lambda_f(r^2)$; some care is required as we do not have absolute convergence. This can be handled replacing $(\ell m)^{-1}$ and $r^{-1}$ by $(\ell m)^{-s}$ and $r^{-s}$, and then taking the limit as $s$ tends to 1 from above. We do not need absolute convergence of the series to justify the limit; it is permissible as the Dirichlet series are continuous in the region of convergence and all sums at $s=1$ exist as the sum of the coefficients grow sub-linearly.


On one side we have
\begin{align*}
\sum_{(m,N)=1}m^{-1}\Delta_{k,N}^\star(m^2,n)&\ = \ \sum_{(m,N)=1}m^{-1}\sum_{f\in H^\star_k(N)} \frac{\lambda_f(m^2)\lambda_f(n)Z_N(1,f)}{Z(1,f)}\\
&\ = \ \sum_{f\in H^\star_k(N)}\frac{\lambda_f(n)}{Z(1,f)}\sum_{(m,N)=1}\sum_{\ell\mid N^\infty}(\ell m)^{-1}\lambda_f(\ell^2)\lambda_f(m^2)\\
&\ = \ \sum_{f\in H^\star_k(N)}\frac{\lambda_f(n)}{Z(1,f)}\sum_{r\ge 1}r^{-1}\lambda_f(r^2)\\
&\ = \ \sum_{f\in H^\star_k(N)}\lambda_f(n)\\
&\ = \ \Delta_{k,N}^\star(n). \eeqno
\end{align*}
On the other hand we have, using \eqref{eq:postinversion}, for $(n,N)=1$,
\begin{align*}
  &\sum_{(m,N)=1}m^{-1}\Delta^*_{k,N}(m^2,n) \\
  &\ = \ \sum_{(m,N)=1}m^{-1}\frac{k-1}{12}\sum_{LM=N}\mu(L)M\prod_{p^2\mid M}\lp\frac{p^2}{p^2-1}\rp^{-1}\sum_{\substack{\ell\mid L^\infty\\(\ell,M)=1}}\frac{1}{\ell}\Delta_{k,M}((m\ell)^2,n) \\
  &\ = \ \frac{k-1}{12}\sum_{LM=N}\mu(L)M\prod_{p^2\mid M}\lp\frac{p^2}{p^2-1}\rp^{-1}\sum_{(m,N)=1}\sum_{\substack{\ell\mid L^\infty\\(\ell,M)=1}}\frac{1}{m\ell}\Delta_{k,M}((m\ell)^2,n) \\
  &\ = \ \frac{k-1}{12}\sum_{LM=N}\mu(L)M\prod_{p^2\mid M}\lp\frac{p^2}{p^2-1}\rp^{-1}\sum_{(m,M)=1}\frac{1}{m}\Delta_{k,M}(m^2,n).\eeqno
\end{align*}
This completes the proof.
\end{proof}


\section{Estimating tails of pure sums}\label{sec:pure}

One might inquire about the convergence of the innermost sum in
Theorem \ref{thm:puresum}. It is assured by the holomorphy of
$L(s,\operatorname{sym}^2 f)$, but is not absolute (see~\cite[p.\,79]{ILS} for a
full discussion).
For this reason, following~\cite[\S2]{ILS}, we begin our work towards the
Density Conjecture by splitting
\begin{equation}\label{spliteq}
\Delta_{k,N}^\star(n) \ = \  \Delta_{k,N}'(n)+\Delta_{k,N}^\infty(n)
\end{equation}
where
\begin{equation}
\Delta_{k,N}'(n)\ := \ \frac{k-1}{12}\sum_{\substack{LM=N\\L\leq X}}\mu(L)M\prod_{p^2\mid M}\lp\frac{p^2}{p^2-1}\rp^{-1}\sum_{\substack{(m,M)=1\\m\leq Y}}m^{-1}\Delta_{k,M}(m^2,n),
\end{equation}
and $\Delta_{k,N}^\infty(n)$ is the complementary sum.
Here $X, Y \geq 1$ are free parameters.

We consider sequences $\A = \{a_q\}$ that satisfy
\begin{equation}\label{sequenceprop}
\sum_{(q,nN)=1}\lambda_f(q)a_q \ \ll \  (nkN)^\varepsilon
\end{equation}
for all $f\in  H_k^\star(M)$ with $M\mid N$ such that the implied constant depends
only on $\varepsilon$.
The sequence we need for our application, given by
\begin{equation}\label{eq:aqsec}
  a_q\ = \ p^{-1/2}\log p\qquad\text{if }q=p\leq Q,\end{equation}
and $a_q=0$ elsewhere, satisfies this property provided $\log Q\ll\log kN$, assuming GRH for $L(s,f)$;
see~\cite[p.\,80]{ILS} for more details. The lemma below is a straightforward modification of Lemma 2.12 of \cite{ILS} to the case of general $N$.

\begin{lemma}\label{lem:infest}
Suppose $(n,N)=1$ and that $\A$ satisfies \eqref{sequenceprop}. Then
\begin{equation}
\begin{split}
\sum_{(q,nN)=1}\Delta_{k,N}^\infty(nq)a_q&\ \ll \  kN(X^{-1} + Y^{-1/2})(nkNXY)^\varepsilon.
\end{split}
\end{equation}
\end{lemma}

\begin{proof} Suppose $(q,nN)=1$. We use Theorem \ref{thm:puresum} and the various definitions of the sums of the $\lambda_f$'s to obtain the expansion for the complementary sum $\Delta_{k,N}^\infty(nq)$, and then sum this weighted by $a_q$ over $q$ relatively prime to $nN$. We simplify some of the resulting sums by grouping them with Lemma \ref{DeltamnRelPrime}. Thus
\begin{align*}
\Delta_{k,N}^\infty(nq)\ = \ &  \sum_{\substack{KLM=N\\ L>X}}\mu(L)\sum_{f\in  H_k^\star(M)}\lambda_f(nq)
\\& \qquad\qquad+ \sum_{\substack{KLM=N\\ L\leq X}}\mu(L)\sum_{f\in H_k^\star(M)}\lambda_f(nq)R_f(KM;Y)\eeqno
\end{align*}
where
\begin{equation}
R_f(KM;Y)\ := \  \frac{Z_{KM}(1,f)}{Z(1,f)}\sum_{\substack{(m,KM)=1\\m>Y}}m^{-1}\lambda_f(m^2).
\end{equation}
By the Riemann Hypothesis for $L(s, \operatorname{sym^2}(f))$ we have
\begin{equation}
R_f(KM;Y)\ \ll \  Y^{-1/2}(kKMY)^\varepsilon.
\end{equation}
Combining this fact with the Deligne bound for $\left|\lambda_f(n)\right|$, we have
\begin{align*}
\sum_{(q,nN)=1}\Delta_{k,N}^\infty(nq)a_q \ = \ &  \sum_{\substack{KLM=N\\L>X}}\mu(L)\sum_{f\in H_k^\star(M)}\lambda_f(n)\sum_{(q,nN)=1}\lambda_f(q)a_q
\\& + \sum_{\substack{KLM=N\\L\leq X}}\mu(L)\sum_{f\in H_k^\star(M)}\lambda_f(n)R_f(KM;Y)\sum_{(q,nN)=1}\lambda_f(q)a_q
\\ \ \ll \  &  \sum_{\substack{KLM=N\\L>X}}|\mu(L) H_k^\star(M)|\tau(n)(nkN)^\varepsilon
\\& +\sum_{\substack{KLM=N\\ L\leq X}}|\mu(L) H_k^\star(M)|\tau(n)Y^{-1/2}(kKMY)^\varepsilon(nkN)^\varepsilon
\\ \ \ll \  & \sum_{\substack{KLM=N\\L>X}}|\mu(L)|\lp\frac{k-1}{12}\rp \varphi(M)\tau(n)(nkN)^\varepsilon
\\&+\sum_{\substack{KLM=N\\ L\leq X}}|\mu(L)|\lp\frac{k-1}{12}\rp\varphi(M)\tau(n)Y^{-1/2}(kKMY)^\varepsilon(nkN)^\varepsilon
\\ \ \ll \   & \sum_{\substack{KLM=N\\L>X}}k\tfrac{N}{X}(nkN)^\varepsilon +\sum_{\substack{KLM=N\\ L\leq X}} kNY^{-1/2}(kKMY)^\varepsilon(nkN)^\varepsilon
\\ \ \ll \  &\  kN(X^{-1} + Y^{-1/2})(nkNXY)^\varepsilon. \eeqno
\end{align*}
This establishes the lemma.
\end{proof}

We now substitute the Petersson formula (Proposition~\ref{prop:petersson}) for
each instance of $\Delta_{k,M}(m^2,n)$ to obtain an exact formula for
$\Delta_{k,N}'(n)$ in terms of Kloosterman sums; this is a generalization of Proposition 2.12 of \cite{ILS}.

\begin{proposition}\label{prop:peterssonondelta}
Suppose $(n,N) = 1$. Then
\begin{align*}
  \Delta_{k,N}'(n)
  \ = \ &\ \delta_{Y}(m^2,n)\frac{k-1}{12}n^{-1/2}\sum_{\substack{LM=N\\L\leq X}}\mu(L)M\prod_{p^2\mid M}\lp\frac{p^2}{p^2-1}\rp^{-1} \\
   &\ +\frac{k-1}{12}\sum_{\substack{LM=N\\L\leq X}}
   \mu(L)M\prod_{p^2\mid M}\lp\frac{p^2}{p^2-1}\rp^{-1} \\
   &\qquad\qquad\times\sum_{\substack{(m,M)=1\\m\leq Y}}m^{-1}2\pi i^k
   \sum_{c\equiv 0 \pmod M}c^{-1}S(m^2,n;c)J_{k-1}\lp\frac{4\pi m\sqrt{n}}{c}\rp,
\end{align*}
where
\begin{equation}
\delta_Y(m^2,n) \ = \  \begin{cases}
1&\text{{\rm if} }n=m^2\text{ {\rm and} } m\leq Y,\\
0 &\text{{\rm otherwise}.}
\end{cases}
\end{equation}
\end{proposition}

We recover the bounds for $|H_k^\star(N)|$ given by Martin in ~\cite[Theorem 6(c)]{Martin}, and immediately obtain the following result (for completeness the calculation, which is standard, is given in the appendix).  

\begin{proposition}\label{prop:Hcard}
We have that as $kN\ra\oo$
\begin{equation}
  \frac{k-1}{12}\varphi(N)\prod_p\pez{1-\frac{1}{p^2-p}}+O\pez{(kN)^{2/3}}
  \ \le \  \abs{H_k^\star(N)}
  \ \le \  \frac{k-1}{12}\varphi(N)+O\pez{(kN)^{2/3}}.
\end{equation}
\end{proposition}


\section{The Density Conjecture for $H_k^\star(N)$}\label{sec:densityconj}

Fix some $\phi\in \mathscr{S}(\R)$ with $\widehat{\phi}$ supported in $(-u,u)$.
We reprise some basic definitions from the introduction.

To a holomorphic newform $f$, we associate the $L$-function
\begin{equation}
L(s,f)\ = \ \sum_{1}^\infty \lambda_f(n)n^{-s}.
\end{equation}
Assuming the Riemann Hypothesis for $L(s,f)$, we can write its non-trivial zeros as
\begin{equation}
\varrho_f\ = \ \frac{1}{2}+i\gamma_f,
\end{equation}
where $\gamma_f\in \mathbb{R}$.
We are interested in the one-level densities of low-lying zeroes. We recall the
definition of $D_1(f;\phi)$ in~\eqref{one level density}:
\begin{equation}
D_1(f;\phi)\ = \ \sum_{\gamma_f}\phi\pez{\frac{\gamma_f}{2\pi}\log c_f},
\end{equation}
where $c_f$ is the analytic conductor of $f$ which in our case is $k^2N$.
We also introduce a scaling parameter $R$ which we take to satisfy $1 < R \asymp k^2N$.

Iwaniec, Luo, and Sarnak \cite[\S4]{ILS} establish that for $f\in  H_k^\star(N)$,
\begin{equation}\label{eq:D=}
D_1(f;\phi)\ = \ E(\phi)-P(f;\phi)+O\pez{\frac{\log \log kN}{\log R}}
\end{equation}
where
\begin{equation}
E(\phi)\ = \ \widehat{\phi}(0)+\frac{1}{2}\phi(0)
\end{equation}
and
\begin{equation}
P(f;\phi)\ = \ \sum_{p\nmid N}\lambda_f(p)\widehat{\phi}\pez{\frac{\log p}{\log R}}\frac{2\log p}{\sqrt{p}\log R}.
\end{equation}
Note that their argument does not depend on $N$ being square-free.
The Density Conjecture concerns the average over $ H_k^\star(N)$, so we consider the
sum
\begin{equation}
\mathscr{B}_k^\star(\phi)\ = \ \sum_{f\in  H_k^\star(N)} D_1(f;\phi).
\end{equation}
Although our main goal is to investigate the behavior as $N\rightarrow\infty$, we keep the notation $\mathscr{B}_k^\star(\phi)$ to remain consistent with \cite{ILS}.
Substituting \eqref{eq:D=} into the above we find that
\begin{equation}
\mathscr{B}_k^\star(\phi)\ = \ \abs{H_k^\star(N)}E(\phi)-\sP^\star_k(\phi)+O\pez{\abs{H_k^\star(N)}\frac{\log\log kN}{\log R}}
\end{equation}
where
\begin{equation}
\sP_k^\star(\phi)\ = \ \sum_{p\nmid N} \Delta^\star_{k,N}(p)\widehat{\phi}\pez{\frac{\log p}{\log R}}\frac{2\log p}{\sqrt{p}\log R}.
\end{equation}
In order to establish that as $kN\rightarrow \infty$ that the main term of
$\mathscr{B}_k^\star(\phi)/\abs{H_k^\star(N)}$ is $E(\phi)$, we need to establish
that $\sP_k^\star(\phi)=o(k\varphi(N))$. This is sufficient because
$\abs{H_k^\star(N)}\asymp k\varphi(N)$, as we showed in Proposition~\ref{prop:Hcard}.

Using the decomposition \eqref{spliteq} we can now write
\begin{equation}\label{eq:Pstar}
\begin{aligned}
\sP_k^\star(\phi)&\ = \ \sum_{p\nmid N} \pez{\Delta'_{k,N}(p)+\Delta^\infty_{k,N}(p)}\widehat{\phi}\pez{\frac{\log p}{\log R}}\frac{2\log p}{\sqrt{p}\log R}.
\end{aligned}
\end{equation}
We first bound
 \begin{equation}
 \sum_{p\nmid N} \Delta^\infty_{k,N}(p)\widehat{\phi}\pez{\frac{\log p}{\log R}}\frac{2\log p}{\sqrt{p}\log R}.
 \end{equation}

Let $a_q$ be as in~\eqref{eq:aqsec} for $q\leq R^u$ and $0$ for $q>R^u$ (the latter is due to
the appearance of $\widehat\phi$, which is zero for $P>R^u$). We see that this sequence satisfies the condition
on $Q$ in the definition~\eqref{eq:aqsec}, and since $\phi$ is of Schwartz class, we
may apply Lemma~\ref{lem:infest} with $X=Y=(kN)^\delta$ for small positive $\delta$
to find
\begin{equation}\label{eq:deltainfs}
 \sum_{p\nmid N} \Delta^\infty_{k,N}(p)\widehat{\phi}\pez{\frac{\log p}{\log R}}\frac{2\log p}{\sqrt{p}\log R}\ \ll \   kN(X^{-1} + Y^{-1/2})(kNXY)^\varepsilon
\ = \ o(k\varphi(N)).
\end{equation}
Next we must estimate the other term from \eqref{eq:Pstar},
\begin{equation}
\mathscr M_k^\star(\phi) \ := \
\sum_{p\nmid N} \Delta'_{k,N}(p)
\widehat{\phi}\pez{\frac{\log p}{\log R}}\frac{2\log p}{\sqrt{p}\log R}.
\end{equation}
To begin, define
\begin{equation}
Q^\star_{k;N}(m;c)\ = \ 2\pi i^k\sum_{p\nmid N} S(m^2,p;c)J_{k-1}\pez{\frac{4\pi m}{c}\sqrt{p}}\widehat{\phi}\pez{\frac{\log p}{\log R}}\frac{2\log p}{\sqrt p\log R}.
\end{equation}
Then we apply Lemma \ref{prop:peterssonondelta} to each instance of $\Delta'_{k,N}$.
Note that the first term in Lemma \ref{prop:peterssonondelta} disappears because $p$
is never a square. Then, moving the initial summation over $p\nmid N$ into the
expression, we can rewrite in terms of $Q_{k;N}^\star(m;c)$:
\begin{equation}
\mathscr M_k^\star(\phi)\ = \
\frac{k-1}{12}\sum_{\substack{LM=N\\L\le X}}\mu(L)M
\prod_{p^2\mid M}\lp\frac{p^2}{p^2-1}\rp^{-1}
\sum_{\substack{(m,M)=1\\m\le Y}}m^{-1}\sum_{c\equiv 0(M)} c^{-1}Q_{k;N}^\star(m;c).
\end{equation}
Iwaniec, Luo, and Sarnak \cite[\S6]{ILS} prove the bound (which still holds for $N$ not square-free)
\begin{equation}\label{eq:Qstarest}
Q_{k;N}^\star(m;c)\ \ll \  \widetilde{\gamma}(z)mP^{1/2}(kN)^\varepsilon( \log 2c)^{-2},
\end{equation}
where $z=4\pi m\sqrt{P}/c$, $P=R^{u'}$ with some $u'<u$, and
$\widetilde{\gamma}(z)=2^{-k}$ if $3z\le k$; this bound appears after their equation (6.17), and uses GRH for Dirichlet $L$-functions (they expand the Kloosterman sums with Dirichlet characters). In order to apply this bound we need to secure $12\pi m P^{1/2}\le kc$, ($3z\leq k$)
so as to satisfy a condition on an estimate for the Bessel function, given in their equation (2.11$'''$).
Noting that $m\le Y$ and $c\ge M\ge N/X$, it suffices to have $12\pi XYP^{1/2}\le kN$.
Taking logarithms, this becomes a condition on $u$, namely
\begin{equation}\label{deltaest}
u\ \le \  \frac{2(1-2\delta)\log(kN)}{\log(k^2N)}.
\end{equation}
For $u$ in this range we can apply the estimate~\eqref{eq:Qstarest} to find
\begin{align*}
  \mathscr M_k^\star(\phi)
  &\ \ll \ \frac{k-1}{12}\sum_{\substack{LM=N\\L\le X}}|\mu(L)|M \\
  &\hspace{1cm}
    \times\prod_{p^2\mid M}\lp\frac{p^2}{p^2-1}\rp^{-1}\sum_{\substack{(m,M)=1\\m\le Y}}
  m^{-1}
  \sum_{c\equiv 0(M)} c^{-1}2^{-k}mP^{1/2}(kN)^\varepsilon( \log 2c)^{-2} \\
  &\ \ll \ \frac{k-1}{12}2^{-k}P^{1/2}(kN)^\varepsilon\sum_{\substack{LM=N\\L\le X}}
  |\mu(L)|M\prod_{p^2\mid M}\lp\frac{p^2}{p^2-1}\rp^{-1}
  \sum_{\substack{(m,M)=1\\m\le Y}}\sum_{c\equiv 0(M)} c^{-1}( \log 2c)^{-2}.\eeqno
\end{align*}
Trivial estimation plus the bound
\begin{equation}
\sum_{c\equiv 0(M)} \frac{1}{c(\log c)^2} \ \ll \  \frac1M
\end{equation}
yields
\begin{equation}\label{eq:deltaprimes}
  \sum_{p\nmid N} \Delta'_{k,N}(p)\widehat{\phi}\pez{\frac{\log p}{\log R}}\frac{2\log p}{\sqrt{p}\log R}\ \ll \  \frac{k-1}{12}2^{-k}P^{1/2}(kN)^\varepsilon XY,
\end{equation}
which is $o(k\varphi(N))$ for $\varepsilon+2\delta<1/2$.

Thus by taking $\delta$ sufficiently small (to satisfy \eqref{deltaest}) and applying the combined estimates for
the completed sums, \eqref{eq:deltainfs} and \eqref{eq:deltaprimes}, we have
established that $\sP^\star_k(\phi)=o(k\varphi(N))$ where the (closed) support of $\widehat\phi$
is contained in $(-u,u)$ and
\begin{equation}\label{eq:supportphi}
u \ < \ \frac{2\log kN}{\log k^2 N},
\end{equation}
which implies Theorem \ref{thm:densitythm}. \hfill $\Box$

\appendix


\section{Estimates for $\abs{H_k^\star(N)}$}\label{sec:cardest}
In this section we recover the main terms of some bounds of
Martin~\cite{Martin} on the cardinality of the set $H_k^\star(N)$.
Since $\lambda_f(1)=1$, we see that $\Delta_{k,N}^\star(1)=\abs{H_k^\star(N)}$. So,
in order to determine the cardinality of $H_k^\star(N)$ it suffices to have an
estimate of $\Delta_{k,N}^\star(1)$.
Taking one term $q=1$ from Lemma~\ref{lem:infest}, we find that
\begin{equation}\label{deltainf}
\Delta_{k,N}^\infty (n)\ \ll \  kN\pez{X^{-1}+Y^{-1/2}}(nkNXY)^\vep.
\end{equation}
Next we turn to evaluating $\Delta_{k,N}'(1)$.
\begin{equation}
\Delta_{k,N}'(n)\ = \ \frac{k-1}{12}\sum_{\substack{LM=N\\L\leq X}}\mu(L)M\prod_{p^2\mid M}\lp\frac{p^2}{p^2-1}\rp^{-1}\sum_{\substack{(m,M)=1\\m\leq Y}}m^{-1}\Delta_{k,M}(m^2,n).
\end{equation}
An application of Weil's bound for Kloosterman sums and a crude bounding of the
Bessel function in the Petersson formula (Proposition~\ref{prop:petersson})
yields an estimate such as~\cite[Corollary 2.2]{ILS}, which we reproduce now.
For any $m,n\ge 1$,
\begin{equation}
\Delta_{k,M}(m,n)\ = \ \delta(m,n)+O\pez{\frac{\tau(M)}{Mk^{5/6}}\frac{(m,n,M)\tau_3((m,n))}{\pez{(m,M)+(n,M)}^{1/2}}\pez{\frac{mn}{\sqrt{mn}+kM}}^{1/2}\log 2mn},
\end{equation}
where the implied constant is absolute. When $n=1$ and $m\le Y$ we find
\begin{equation}
\Delta_{k,M}(m^2,1)\ = \ \delta(m,1)+O\pez{\frac{M^{\varepsilon}m^{1+\varepsilon}}{M^{3/2}k^{4/3}}}.
\end{equation}
In our case this gives that
\begin{equation}\label{deltaprime}
\Delta_{k,N}'(1)\ = \ \frac{k-1}{12}\sum_{\substack{LM=N\\L\leq X}}\mu(L)M\prod_{p^2\mid M}\lp\frac{p^2}{p^2-1}\rp^{-1} +O\pez{\frac{X^{1/2}Y(kNY)^\varepsilon}{N^{1/2}k^{1/3}}
}.\end{equation}
We now turn to the evaluation of
\begin{equation}\label{etais}
  \eta(N)\ := \ \sum_{LM=N}\mu(L)M\prod_{p^2\mid M}\pez{1-\frac{1}{p^2}}.
\end{equation}
Let $M''$ denote the square-full part of $M$ and let $g(M)=M/\zeta_{M''}(2)$, which
is multiplicative. Then $\eta=\mu \star g$ is multiplicative as it is the Dirichlet
convolution of two multiplicative functions and we can compute directly the value
of $\eta$ on prime powers:
\begin{equation}
\eta(p^v)\ = \ \begin{cases}
p\pez{1-\frac{1}{p}}&\text{ if } v=1 \\
p^2\pez{1-\frac{1}{p}-\frac{1}{p^2}}&\text{ if } v=2 \\
p^v\pez{1-\frac{1}{p^2}}\pez{1-\frac{1}{p}}&\text{ if } v>2.
\end{cases}
\end{equation}
It is also useful to establish a bound relating $\eta(N)$ to $\varphi(N)$. By inspection we have that $\eta(N)\le \varphi(N)$. Then as the ratio $\eta(p^v)/\varphi(p^v)$ is minimized when $v=2$ and as
\begin{equation}
\eta(p^2)/\varphi(p^2)\ = \ \frac{p^2-p-1}{p^2-p}\ = \ 1-\frac{1}{p^2-p},
\end{equation}
we find that
\begin{equation}\label{etabounds}
\varphi(N)\prod_p \pez{1-\frac{1}{p^2-p}}\ \le \  \eta(N)\ \le \  \varphi(N).
\end{equation}
Now combining \eqref{deltainf} \eqref{deltaprime} and \eqref{etais}, we find that
\begin{multline}
  \Delta^\star_{k,N}(1)\ = \ \frac{k-1}{12}\eta(N)\pez{1+O\pez{\frac{\tau(N)N}{\eta(N)X}}}
  +O\pez{\frac{X^{1/2}Y(kNY)^\varepsilon}{N^{1/2}k^{1/3}}} \\
  +O\pez{kN\pez{X^{-1}+Y^{-1/2}}(kNXY)^\vep}.
\end{multline}
Taking $X=Y^{1/2}=k^{8/21}N^{3/7}$ then gives
\begin{equation}
\Delta^\star_{k,N}(1)\ = \ \frac{k-1}{12}\eta(N)+O\pez{(kN)^{2/3}},
\end{equation}
which recovers the tight asymptotic bounds given on $\abs{H_k^\star(N)}$ in~\cite[Theorem 6(c)]{Martin}. Combining this with~\eqref{etabounds} we establish the following.
\begin{proposition}[Proposition \ref{prop:Hcard}]
We have that as $kN\ra\oo$
\begin{equation}
  \frac{k-1}{12}\varphi(N)\prod_p\pez{1-\frac{1}{p^2-p}}+O\pez{(kN)^{2/3}}
  \ \le \  \abs{H_k^\star(N)}
  \ \le \  \frac{k-1}{12}\varphi(N)+O\pez{(kN)^{2/3}}.
\end{equation}
\end{proposition}

\end{document}